\documentclass[a4paper,12pt]{amsart}
 \usepackage{amsfonts,mathrsfs}
 \usepackage{amsthm}
 \usepackage{amssymb}
 \usepackage{enumerate}
 \usepackage{tikz-cd}
 \usepackage{MnSymbol}
 \usepackage{mathtools}
 \usepackage{hyperref}
 \usepackage{color}
 \usepackage{subcaption}
 \usepackage{cleveref}
 \usepackage{csquotes}
\usepackage{amsmath}
\usepackage[new]{old-arrows}

 \usepackage{listings}
 \usepackage{tikz-cd}
\usepackage{mathrsfs}

\newcommand\R{\mathbb{R}}

\newcommand\Z{\mathbb{Z}}
\newcommand\N{\mathbb{N}}

\DeclareMathOperator\conv{conv}

\theoremstyle{definition}
\newtheorem{Def}{Definition}[section]
\newtheorem{definition}[Def]{Definition}
\newtheorem{exm}[Def]{Example}
\newtheorem{Exm}[Def]{Example}
\newtheorem{rem}[Def]{Remark}
\theoremstyle{plain}

\newtheorem{proposition}[Def]{Proposition}

\newtheorem{theorem}[Def]{Theorem}
\newtheorem*{thm*}{Theorem}
\newtheorem{lem}[Def]{Lemma}

\newtheorem*{cor*}{Corollary}

\newtheorem*{con*}{Conjecture}

\newcommand{\ceil}[1]{\left\lceil #1 \right\rceil}
\newcommand{\floor}[1]{\left\lfloor #1 \right\rfloor}
\DeclareMathOperator{\Ehr}{Ehr}
\DeclareMathOperator{\aff}{aff}

\makeatletter
\@namedef{subjclassname@2020}{%
  \textup{2020} Mathematics Subject Classification}
\makeatother

\title[Alcoved Polytopes Unimodal $h^*$?]{Do alcoved lattice polytopes have unimodal $h^*$-vector?}
\author{Rainer Sinn}
\address{Augustusplatz 10, Mathematisches Institut, Universit\"at Leipzig, 04109 Leipzig, Germany} 
\email{rainer.sinn@uni-leipzig.de}
\author{Hannah Sjöberg}
\address{Arnimallee 2, Institut für Mathematik, Freie Universit\"at Berlin, 14195 Berlin, Germany}
\begin{document}

\subjclass[2020]{Primary: 52B11, 52B20, 13F55}

\begin{abstract}
We show that $h^*$-vectors of alcoved polytopes $P\subset \R^n$ (of Lie type $\mathcal{A}$) are unimodal if they contain interior lattice points and their facets have lattice distance $1$ to the set of interior lattice points. The maximal possible such distance for general alcoved polytopes is shown to be $\dim(P) - 1$.
A secondary purpose of the paper is to serve as a guide to previous work surrounding unimodality of $h^*$-vectors of alcoved polytopes and related questions.
\end{abstract}
\maketitle

\section*{Introduction}

A finite sequence $(s_1, s_2, \ldots, s_n)$ is called \emph{unimodal} if there exists an index $k \in \{1, \ldots, n\}$ such that $s_1 \le \ldots  \le s_k \ge \ldots \ge s_n$. We are particularly interested in this property for sequences arising from lattice polytopes, which are polytopes in $\R^n$ whose vertices lie in $\Z^n$, namely $h^*$-vectors. Closely related (at least in our case) are face vectors of (regular unimodular) triangulations, namely $h$-vectors.
There is a variety of conjectures and theorems about the unimodality of $h$-vectors and $h^*$-vectors of different objects.
The articles \cite{Sta89} by Stanley, \cite{Bre94} by Brenti, and \cite{Bra14} by Br\"and\'en are excellent surveys on unimodality in combinatorics. We focus particularly on unimodality in Ehrhart theory for which we refer to the survey \cite{Bra16} by Braun.

The $h^*$-polynomial of a lattice polytope is derived from its Ehrhart series. For $t\in \N$, write $L_P(t) = |(tP)\cap \Z^n|$ for the number of lattice points in the dilate $tP$ of the lattice polytope $P\subset \R^n$. Its Ehrhart series is the generating function of this lattice count, i.e.
\[
	\Ehr_P(z) = 1 + \sum_{t\geq 1} L_P(t) z^t.
\]
It turns out that this series is actually a rational function with denominator $(1-z)^{d+1}$, where $d$ is the dimension of $P$. The numerator is a polynomial of degree at most $d$ and its coefficient vector is the $h^*$-polynomial of $P$, more precisely
\[
	\Ehr_P(z) = \frac{h_d^* z^d + h_{d-1}^* z^{d-1} + \ldots + h_1^* z + h_0^*}{(1-z)^{d+1}}.
\]
The $h^*$-vector of $P$ is the vector of coefficients of the $h^*$-polynomial. 
We are interested in unimodality of the $h^*$-vector. 

Our main \Cref{thm:interiorlatticepoints} says that the $h^*$-vector is unimodal for a special class of lattice polytopes called alcoved polytopes of Lie type $\mathcal{A}$ whose facets have lattice distance $1$ to the interior lattice points. We define these notions in \Cref{sec:prelims}. We also show that our proof strategy does not work if we drop the assumption about the lattice distance in \Cref{rem:proofstrategy}. Along the way, we show \Cref{thm:UTunimodal} stating that the $h^*$-vector is unimodal for 
any $d$-dimensional lattice polytope $P$ that admits a unimodular triangulation $\Delta$ such that the induced simplicial complex on $(\partial P)\cap \Z^d$ is a triangulation of $\partial P$. 

Here is some historical context of this result with other classes of lattice polytopes with unimodal $h^*$-vectors.
\begin{thm*}[{Hibi~\cite{Hi92book}}]
	Reflexive lattice polytopes up to dimension $5$ have unimodal $h^*$-vectors.	
\end{thm*}

This theorem does not hold for higher dimensions. Mustaţă and Payne \cite[Ex. 1.1]{MP05}, \cite[Thm. 1.4]{Pa08} showed that there exist reflexive polytopes (even simplices) of all dimensions greater than $5$ without unimodal $h^*$-vectors.

Bruns and Römer showed that if a reflexive (or more generally a Gorenstein) polytope has a regular unimodular triangulation, then its $h^*$-vector is unimodal.
\begin{thm*}[{Bruns \& Römer~\cite{BrRo07}}]
	Gorenstein lattice polytopes with regular unimodular triangulation have unimodal $h^*$-vectors.	
\end{thm*}

It has been conjectured that the property of having a unimodular triangulation can be weakened even more to the more general condition of having the \textit{integer decomposition property} (IDP).
\begin{con*}[Hibi \& Ohsugi \cite{HO06}]
A lattice polytope which is Gorenstein and IDP has unimodal $h^*$-vector.
\end{con*}

Even more generally, there are no known examples of IDP polytopes without unimodal $h^*$-vectors. The next question or conjecture is part of a conjecture of Stanley's that standard graded Cohen--Macaulay integral domains have unimodal $h$-vectors.

\begin{con*}[{Stanley, see \cite[Question~1.1]{SvL13}}]
	IDP polytopes have unimodal $h^*$-vectors.	
\end{con*}

A special class of polytopes that is conjectured to have 
unimodal $h^*$-vectors is the class of \emph{order polytopes}:
Let $P=(\{p_1, \ldots, p_n\},\preccurlyeq)$ be a finite poset.
The \emph{order polytope} $\mathcal{O}(P)\subset \R^{n}$
is defined by the inequalities:
\begin{align*}
	0 \le x_i \le 1  &\text{ for all } i \in \{1, \ldots, n\},\\
	x_i \le x_j  &\text{ if } p_i \preccurlyeq p_j.
\end{align*}
Order polytopes have regular unimodular triangulations.
We will see an example of such a triangulation for the more general class of alcoved polytopes in Definition~\ref{def:alcoved}.

Alcoved polytopes have regular unimodular triangulations (but they are not necessarily Gorenstein or reflexive). 
Here is a general result in the direction of unimodality of $h^*$-vectors that gives us about half of the inequalities in our main theorem.
\begin{thm*}[{Hibi \& Stanley, see Athanasiadis~\cite[Theorem~1.3]{Ath04}}]
Let $P$ be a $d$-dimensional lattice polytope with a regular unimodular triangulation.
Then:
\begin{align*}
\begin{cases}
&h^*_i(P)\ge h^*_{d+1-i}(P) \hspace{82pt} \text{ for } 1\le i \le \floor{\frac{d+1}{2}},\\
& h^*_{\floor{\frac{d+1}{2}}}(P)\ge \ldots \ge h^*_{d-1}(P)\ge h^*_d(P),\\
& h^*_i(P)\le \binom{h^*_1(P) +i-1}{i} \hspace{80pt} \text{ for } 0 \le i \le d.
\end{cases}
\end{align*}
\end{thm*}

\section{Preliminaries}\label{sec:prelims}
We write $[d]$ for the set $\{1,2,\ldots,d\}$.
\subsection{Alcoved Polytopes}
\begin{definition}\label{def:alcoved}
A hyperplane coming from an affine Coxeter arrangement of type $\mathcal{A}_d$ is a hyperplane of the form
\[
	H_d(i,j,k)=\{
	\begin{pmatrix}
	y_1\\ \vdots \\ y_d
	\end{pmatrix}
	\in \R^d\mid y_i-y_j= k_{ij}\}
\]
 for some $k_{ij}\in \Z$, $i,j \in \{0, \ldots, d\}$ and $y_0=0$.
 We will call such hyperplanes \emph{alcove hyperplanes}.
 
A $d$-dimensional polyhedron is called an \emph{alcoved polyhedron} (of Lie type $\mathcal{A}$) if all of its facet-defining hyperplanes are alcove hyperplanes.
If $P$ is bounded it is called an \emph{alcoved polytope} of Lie type $\mathcal{A}$.
In the following by an alcoved polytope we always mean an alcoved polytope of Lie type $\mathcal{A}$.
\end{definition}

The $\mathcal{H}$-description of an alcoved $d$-polyhedron $P$ with $m$ facets can be written as
\[P=\{x\in \R^d \mid Mx\le b\},\]
for $b\in \Z^m$ and where $M$ is an $(m\times d)$-matrix with row vectors $a_k \in \{e_i\colon i\in [d]\} \cup \{e_i - e_j\colon i,j\in [d]\}$ for all $k \in \{1, \ldots m\}$.
This matrix $M$ is always a totally unimodular matrix, i.e. every minor of $M$ is in $\{0,\pm 1\}$. In particular, this implies that alcoved polytopes are lattice polytopes (see \cite[Sect.~7.1]{Bar17}).

We obtain the \emph{alcoved triangulation} of an alcoved $d$-polytope $P$ by subdividing it by all alcove hyperplanes $H_d(i,j,k)$ ($i,j\in \{0\}\cup [d]$, $k\in \Z$), see \cite[Section 2.3]{LaPo07}.
The simplices of the alcoved triangulation are called \emph{alcoves}. 
The unimodular triangulation is regular, for example via the lifting function
 \[
 	y_{d+1}=\sum_{i=1}^d y_i^2 +\sum_{\{i,j\}\in \{1, \ldots, d\}} (y_i-y_j)^2.
 \]
  
Examples of alcoved polytopes include concrete realizations of hypersimplices and order polytopes (see \cite{LaPo07} for more details and more examples of alcoved polytopes).
\begin{Exm}[Hypersimplex]
	The hypersimplex $\Delta_{d,k}$ is usually defined as the set of vectors $x = (x_1,x_2,\ldots,x_d)$ in $[0,1]^d\subset \R^d$ whose coordinates sum to $k$, i.e.~$\sum_{i=1}^d x_i = k$. After the affine change of coordinates $z_j = \sum_{i=1}^j x_i$, the hypersimplex can be defined by the inequalities 
	\[
		\begin{array}{cl}
			0 \leq z_1 \leq 1 &  \\
			k \leq z_{d} \leq k & \\
			0 \leq z_i - z_{i-1} \leq 1 & \text{ for all } i\in [d]
		\end{array}
	\]
	in the affine hyperplane $\{z\colon z_d = k\} \subset \R^d$. This is a description as an alcoved polytope.
\end{Exm}

\begin{Exm}[Order polytopes]
	Let $P = (\{p_1,p_2,\ldots,p_n\},\preceq)$ be a finite poset. The \emph{order polytope} $\mathcal{O}(P)\subset \R^d$ is the polytope defined by the inequalities
	\begin{align*}
		0 \leq x_i \leq 1 & \text{ for all }i\in [n] \\
		x_i\leq x_j & \text{ whenever } p_i \preceq p_j.
	\end{align*}
	The order polytope corresponding to an anti-chain is the $0/1$-cube. The order polytope of a chain is a simplex.
\end{Exm}

Since an alcoved polytope can only have a finite number of admissible facet normals, there exists a unique alcoved polytope of minimal volume among all alcoved polytopes containing the origin in the interior, which we call $Q_d$ throughout the paper. This polytope is obtained by taking the intersection of all facet-defining half-spaces that are defined by alcoved hyperplanes and contain the origin in the interior:
\begin{definition}\label{def:Qd}
Let $Q_d$ denote the alcoved polytope of minimal volume among all $d$-dimensional alcoved polytopes that have the origin in the interior:
\[
	Q_d\coloneqq
	\left\{
	(y_1, \ldots, y_d)
	\in \R^d\mid y_i-y_j\le 1 \text{ for } 0\le i,j\le d, y_0=0 \right\}.
\]
\end{definition}  
 
\begin{figure}[htbp]
\centering
\begin{subfigure}{.6\textwidth}
  \centering
  \includegraphics[width=.4\linewidth]{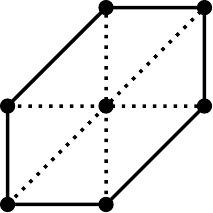}
  \caption{Alcoved triangulation of the hexagon $Q_2$}
  \label{fig:Q2}
\end{subfigure}%
\begin{subfigure}{.5\textwidth}
  \centering
  \includegraphics[width=.6\linewidth]{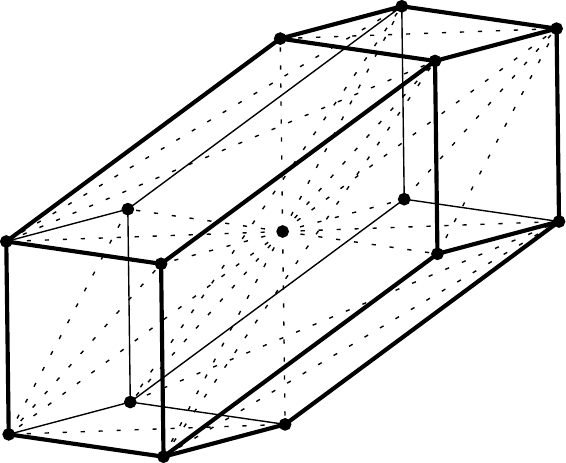}
  \caption{Alcoved triangulation of $Q_3$}
  \label{fig:Q3}
\end{subfigure}
\caption{Examples of the polytope $Q_d$}
\label{fig:Qd}
\end{figure}
 
The next proposition gives some idea about the combinatorial and geometric properties of the polytope $Q_d$.
\begin{proposition} \label{prop:qd}

\

\begin{enumerate}[(i)]
\item  $Q_d$ is centrally symmetric.  \label{prop:cs}
\item $Q_d$ has $2\binom{d+1}{2}$ facets.\label{prop:facetsQd}
\item $Q_d$ is the convex hull of the union of the cubes $[-1,0]^d$ and $[0,1]^d$.
This is a polytope that contains  $2^{d+1}-1$ lattice points. It has one interior lattice point and all other $2^{d+1}-2$ lattice points are vertices.
\label{convhull}
\item The polytope $Q_d$ is a projection of the $(d+1)$-dimensional unit cube $[0,1]^{d+1}$.
Moreover, $Q_d$ has the same $h^*$-vector as the unit cube $[0,1]^{d+1}$.\label{prop:cube}
\end{enumerate}
\end{proposition}
 
\begin{proof}
(\ref{prop:cs}) Central symmetry follows immediately from the hyperplane description. 

(\ref{prop:facetsQd})
To see that $Q_d$ has $2\binom{d+1}{2}$ facets, observe that all $2\binom{d+1}{2}$ hyperplanes in the hyperplane description of $Q_d$ are irredundant: The point with coordinates  
$x_i=1$, $x_j=-1$ and $x_k=0$ for all $k\in\{0, \ldots d\}\setminus\{i,j\}$ is not contained in $Q_d$, but it is contained in the polyhedron obtained by removing $\{x \in R^d \mid x_i-x_j\le 1\}$ from the hyperplane description of $Q_d$.

(\ref{convhull})
The cubes $[-1,0]^d$ and $[0,1]^d$ have $2^d$ lattice points each. The only lattice point in common is $\textbf{0}$, so together they have $2^{d+1}-1$ lattice points.

It follows readily from the hyperplane description of $Q_d$ that all vertices from the cubes $[-1,0]^d$ and $[0,1]^d$ are contained in $Q_d$.
So the convex hull of $[-1,0]^d$ and $[0,1]^d$ is contained in $Q_d$.
Since $Q_d$ is an alcoved polytope, and hence a lattice polytope,
in order to show that $Q_d$ is equal to the convex hull
it suffices to show that  the vertices of the cubes $[-1,0]^d$ and $[0,1]^d$ are the only lattice points contained in $Q_d$.
From the inequalities $x_i\le 1$ and $-x_i\le 1$ it follows that all coordinates of the points in $Q_d$ lie between $-1$ and $1$. Because of the inequality $x_i -x_j \le 1$, 
no point in $Q_d$ can contain both $x_i=1$ and $x_j=-1$ as coordinates. This shows that all lattice points in $Q_d$ are vertices from $[-1,0]^d$ or from $[0,1]^d$,
and hence that $Q_d$ is the convex hull of $[-1,0]^d$ and $[0,1]^d$.
The point $\textbf{0}$ is the unique interior lattice point of $Q_d$,
all other $2^{d+1}-2$ lattice points are vertices:
No point with coordinates in $\{0,1\}$ or $\{-1,0\}$ besides the point $\textbf{0}$ can be written as a convex combination of the other points.

(\ref{prop:cube}) $Q_d$ is the image of the unit cube $[0,1]^{d+1}$ under the projection $\varphi:\R^{d+1} \rightarrow \R^d$, $e_i\mapsto e_i$
for $i=1, \ldots, d$ and  $e_{d+1}\mapsto -e_1-\ldots -e_d$ with totally unimodular transformation matrix
\[\left(
\begin{array}{rc}
    \mbox{\Large $I_d$} &   \begin{matrix}  -1 \\ \vdots\\ -1 \end{matrix} 
\end{array}
\right).\]
The $(d+1)$-simplices in the alcoved triangulation of the $(d+1)$-cube 
are of the form $\conv \{0, e_{i_1}, e_{i_1}+e_{i_2}, \ldots, e_{i_1}+e_{i_2}+\ldots +e_{i_{d+1}}\}$.

They are mapped to $\conv \{0, e_{i_1}, e_{i_1}+e_{i_2}, \ldots,e_{i_1}+e_{i_2}+\ldots +e_{i_d}\}$, where $i_j\in\{0, \ldots, d\}$ with $e_0\coloneqq -e_1-\ldots -e_d$. These are the $d$-simplices of the alcoved triangulation of $Q_d$.
Intersections of $(d+1)$-simplices in the alcoved triangulation of the $(d+1)$-cube are mapped to the intersections of 
the corresponding $d$-simplices of $Q_d$. So if $\Delta_1, \ldots,  \Delta_{(d+1)!}$ is a shelling order of the $(d+1)$-simplices of the alcoved triangulation of $[0,1]^{d+1}$, then $\varphi(\Delta_1), \ldots,  \varphi(\Delta_{(d+1)!})$ is a shelling order of the $d$-simplices of the alcoved triangulation of $Q_d$.
The triangulations have therefore the same $h$-vectors.
Since alcoved triangulations are unimodular triangulations, it follows that $Q_d$ has  the same $h^*$-vector as $[0,1]^{d+1}$.
\end{proof}

\subsection{\texorpdfstring{$h^*$-vectors of lattice polytopes}{h{*}-vectors of lattice polytopes}}
We briefly introduce basic concepts from Ehrhart theory. 
For more details, see \cite{BeRo07}.

A \emph{lattice polytope} is a polytope in $\R^d$ whose vertices all have integer coordinates, i.e.~lie in $\Z^d$, which is the fixed lattice in $\R^d$ for us. For a lattice polytope $P\subset\R^d$ and a positive integer $t\in\N$, let $L_P(t)$ denote the number of lattice points in the $t$-th dilate of $P$, i.e.
\[L_P(t)\coloneqq |(tP)\cap\Z^d|.\]
The \emph{Ehrhart series} of $P$ is the corresponding generating function
\[\Ehr_P(z)\coloneqq 1 + \sum_{t\in \N_{>0}}L_P(t)z^t.\]
\begin{theorem}[Ehrhart {\cite[Thm.~2]{Ehr62}}]
\label{thm:ehrhart}
Let $P\subset\R^d$ be a lattice polytope. Then there exist complex numbers $h^*_i$ such that
\[\Ehr_P(z) = \frac{h_0^*(P)+h_1^*(P)z+ \ldots+h_d^*(P)z^d}{(1-z)^{d+1}}.\]
\end{theorem}
A corollary of this theorem is that $L_P(t)$ can be expressed as a rational polynomial of degree at most $d$ in the variable $t$, i.e. there exist rational numbers $q_0(P), q_1(P), \ldots, q_{d-1}(P),q_d(P)$ such that
\[L_P(t)= q_0(P) + q_1(P)t+ \ldots + q_{d-1}(P)t^{d-1}+q_d(P)t^d\]
for all $t\in \N_{>0}$.
This polynomial is called the \emph{Ehrhart polynomial} of $P$.
The vector $h^*(P)\coloneqq (h_0^*(P), h_1^*(P), \ldots, h_d^*(P))$ of coefficients of the numerator of the Ehrhart series in its expression as a rational function in \Cref{thm:ehrhart} is called the $h^*$-\emph{vector} ($h$-star-vector) of $P$.

Stanley later showed that the entries of the $h^*$-vector are not just complex numbers, see \cite[Thm.~2.1]{Sta80b}.
\begin{theorem}\label{thm:Stanleynonneg}
\hypertarget{Stanleynonneg}{The} entries of the $h^*$-vector of a lattice polytope are non-negative integers.
\end{theorem}

Some entries of the $h^*$-vector are known to have a nice combinatorial or geometric interpretation (see \cite[Section~1]{HT09}), for instance
\[
h^*_0(P) = 1,  h^*_1(P) =|P\cap \Z^d|-(d+ 1), \text{ and } h^*_d(P) =|\text{int}(P)\cap \Z^d|.
\]

\subsection{Reflexive and Gorenstein lattice polytopes}
The \emph{lattice distance} between  a hyperplane $H = \{x\in\R^d \colon m^tx = b\}$ for some vector $m\in\Z^d$ and $b\in \Z$ and a lattice point $p\in\Z^d$ is $0$ if $p\in H$. Otherwise it is $n+1$, where $n$ is
the number of parallel translates of $H$ through lattice points that are strictly between $p$ and $H$.
In particular, a hyperplane $H$ and a point $p\notin H$ have lattice distance $1$ if there are no lattice points that are strictly between $H$ and the hyperplane parallel to $H$ containing $p$.
The lattice distance between a facet $F$ and a lattice point is the lattice distance between the hyperplane $\aff(F)$ and $p$.

A lattice polytope with $0$ in its interior is called \emph{reflexive} if its polar is also a lattice polytope. Often, lattice polytopes are called reflexive if an appropriate translation by a lattice point is reflexive. This is then the same as saying that it has a unique interior lattice point and all facets have lattice distance $1$ from the interior lattice point. A lattice polytope $P\subset \R^n$ is \emph{Gorenstein} of index $k$ if $kP$, the $k$-th dilate of $P$, is a reflexive polytope for some $k\in \N_{>0}$.
\begin{theorem}\cite[Hibi]{Hi91}
A lattice $d$-polytope $P$ in $\R^d$ is reflexive (up to unimodular equivalence) if and only if 
its $h^*$-vector is symmetric, i.e.:
\[h^*_i=h^*_{d-i} \text{ for } 0\le i \le d.\] 
\end{theorem}

A lattice polytope $P$ is said to possess the \emph{integer-decomposition property} (IDP)
if for all $k\in \N_{>0}$, every lattice point in $kP$ can be written as a sum of $k$ lattice points of $P$.
Polytopes which possess the IDP are called \emph{IDP polytopes}, for short.

\subsection{Triangulations}
Next we look at some triangulations of lattice polytopes. A good reference for triangulations is \cite{DLRS10}.

A \emph{triangulation} of a point configuration $\mathcal{A}\subset\R^d$ is a simplicial complex with vertex set in $\mathcal{A}$ that covers $\conv(\mathcal{A})$.
By a triangulation of a lattice polytope $P$ we always mean a triangulation of the point configuration $P\cap \Z^d$.

A full-dimensional lattice simplex $S$ in $\R^d$ with vertices $v_0 , \ldots , v_n$ is called 
a \emph{unimodular simplex}
if the vectors $v_n -v_0, v_{n-1} -v_0 , \ldots , v_1 -v_0$ form a basis for $\Z^d$.
All unimodular $d$-dimensional lattice simplices have the same volume $\frac{1}{d!}$.
The volume of lattice polytopes is often normalized by the factor $d!$, so that a unimodular simplex is said to have \emph{normalized volume} $1$. A triangulation of a lattice polytope is a \emph{unimodular triangulation} if all its simplices are unimodular. The existence of a unimodular triangulation for a lattice polytope $P\subset \R^d$ implies that $P$ has the IDP, see \cite[Thm.~1.2.5]{HPPS17}.

A triangulation $\Delta(P)$  of a $d$-polytope $P$ is called a \emph{regular triangulation}  if the following conditions hold:
$P$ is the image $\pi(Q)$ of a polytope $Q\subset \R^{d+1}$ 
under the projection to the first $d$ coordinates:
\begin{align*}
\pi:  \R^{d+1} &\longrightarrow \R^d\\
 \begin{pmatrix}x\\ x_{d+1}\end{pmatrix}&\longmapsto x,
\end{align*}
and $\Delta(P)$ is the image of all \emph{lower faces} of $P$ under the projection $\pi$.
A \emph{lower face} $F$ is a face of $Q$ whose outer normal vector has a negative last coordinate.

Lattice polytopes which have a unimodular triangulation have particularly nice properties.
The following proposition is an example.
It allows us to reduce all questions about $h^*$-vectors of lattice polytopes with unimodular triangulations to questions about the $h$-vectors of the triangulation.
\begin{proposition}[Betke \& McMullen {\cite{BM85}}]\label{BetkeMcMullen}
For any lattice polytope $P$ which has a unimodular triangulation $\Delta(P)$,
the $h^*$-vector of $P$ is equal to the $h$-vector of $\Delta(P)$.
\end{proposition}

This observation is the basis of the proof of our main \Cref{thm:interiorlatticepoints}.

\subsection{\texorpdfstring{Stanley-Reisner Theory and $h^*$-vectors}{Stanley-Reisner Theory and h{*}-vectors}}
We explain some notions from Stanley--Reisner theory, which is an algebraic approach to simplicial complexes.
For more details see \cite{Sta96}, \cite{BrHe09}, or \cite{MS05}.

Let  $\Delta$ be an abstract simplicial complex with vertices $x_1, \ldots, x_n$. Let $K$ be a field and $K[X_1, \ldots, X_n]$ be the polynomial ring over $K$ where variable $X_i$ corresponds to vertex $x_i$.

The \emph{Stanley-Reisner  ideal}  of
$\Delta$  is  the  squarefree  monomial ideal $I_{\Delta}$ of $K[X_1, \ldots, X_n]$
generated by all the square-free monomials $X_{i_1}X_{i_2}\ldots X_{i_s}$
corresponding to 
 the non-faces $\{x_{i_1},x_{i_2}\ldots ,x_{i_s}\}$ of $\Delta$:
 \[
 X_{i_1}X_{i_2}\ldots X_{i_s} \in I_{\Delta} \text{ if and only if } 
 \{x_{i_1},x_{i_2}\ldots ,x_{i_s}\} \notin \Delta.
 \]
The \emph{face ring} (or \emph{Stanley-Reisner ring}) $K[\Delta]$ of  $\Delta$ 
is  the  quotient of $K[X_1, \ldots, X_n]$
 by  the  Stanley-Reisner  ideal,
\[
K[\Delta]\coloneqq K[X_1, \ldots, X_n]/ I_{\Delta}.
\]

The Stanley--Reisner correspondence allows us to express many combinatorial problems of simplicial complexes in terms of homological algebra. We need the notion of a \emph{Cohen--Macaulay ring}. In general, a commutative Noetherian local ring is called \emph{Cohen--Macaulay} if its depth is smaller than or equal to its Krull dimension. Since we only work with Stanley-Reisner rings here, we can simplify this definition to a characterization of Cohen--Macaulay rings for the case of certain quotient rings:

\begin{proposition}[Hironaka's criterion, see \cite{StanleyUBT}, Prop. 4.1]\label{Hironaka}
\hypertarget{label:Hironaka}{Let} $K$ be an infinite field and let $R\coloneqq K[X_0, \ldots, X_n]/ I$ be the quotient of $K[X_0, \ldots, X_n]$ by a homogeneous ideal $I$. Let $d$ denote the Krull dimension of $R$.
$R$ is a \emph{Cohen--Macaulay ring} if and only if there exist $d$ homogeneous, linear elements $\theta_1, \ldots, \theta_d$ from  $R$ and finitely many elements $\eta_1, \ldots, \eta_n$ from $R$ such that every $p \in R$ has a unique representation as
\[p=\sum^{n}_{i=1}\eta_ip_i(\theta_1, \ldots, \theta_d),\]
where the $p_i(\theta_1, \ldots, \theta_d)$ are elements in $K[\theta_1, \ldots, \theta_d]$.
\end{proposition}

Equivalently, we can say that $R$ is a free $K[\theta_1, \ldots, \theta_d]$-module with basis $(\eta_1, \ldots, \eta_n)$.

The system $\Theta \coloneqq (\theta_1, \ldots, \theta_d)$ is called a \emph{linear system of parameters (l.s.o.p.)}
for $R$. If there exists a l.s.o.p. for a ring $R$, then any generic choice of $\theta_1, \ldots, \theta_d$ will be a l.s.o.p.
The term ``generic'' here refers to elements from a Zariski open subset of $R_1^d$, see~\cite{Ke99}. 

A simplicial complex $\Delta$ is called \emph{Cohen--Macaulay}
over a field $K$ if its Stanley-Reisner ring $K[\Delta]$ is a  Cohen--Macaulay ring.

Reisner's criterion gives a topological characterization of Cohen--Macaulay complexes in terms of their (reduced and simplicial) homology groups.
\begin{definition}
	Let $\Delta$ be a simplicial complex. The \emph{link} of a face $F\in \Delta$ is the set of all faces $G$ that are disjoint with $F$ but $F\cup G$ is a face of $\Delta$, i.e.
	\[
		\mathrm{link}(\Delta,F) = \{ G\in \Delta\colon F\cap G = \emptyset \text{ and } F\cup G\in\Delta\}.
	\]
	The \emph{star} for a face $F\in \Delta$ is the set of all faces $G$ that contain $F$, i.e.
	\[ \mathrm{star}(\Delta,F) = \{ G\in \Delta\colon F\subset G\}. \]
\end{definition}

\begin{proposition}[Reisner's criterion \cite{Reisner}]
A simplicial complex $\Delta$  is Cohen--Macaulay over a field $K$ if and only if for any face $F$ of $\Delta$, 
\[\dim_K(\overset{\sim}{H}_i(\mathrm{link}({\Delta},F);K))=0 \text{ for } i<\dim(\mathrm{link}({\Delta},F)).\]
that is, $\Delta$ is Cohen--Macaulay over $K$ if and only if the homology of each face's link vanishes below its top dimension.
\end{proposition}
In particular, this implies that pure shellable simplicial complexes are Cohen--Macaulay, see \cite[Theorem~13.15]{MS05}. 

If we now have a pure shellable simplicial complex $\Delta$ of dimension $d-1$, then its face ring $K[\Delta]$ has Krull dimension $d$. According to \hyperlink{label:Hironaka}{Hironaka's criterion} we can choose a l.s.o.p. $\Theta = (\theta_1, \ldots, \theta_d)$ and consider the quotient ring $K[\Delta]/\Theta$. The total degree  grading of $K[X_1,\ldots,X_n]$ induces a grading on the quotient so that we can write the face ring as the direct sum 
\[
K[\Delta]/\Theta = (K[\Delta]/\Theta)_0 \oplus (K[\Delta]/\Theta)_1 \oplus \ldots \oplus (K[\Delta]/\Theta)_d,
\]
where $\dim_K (K[\Delta]/\Theta)_i<\infty$ for $i= 0, \ldots, d$.

\begin{proposition}[{see Stanley \cite[Sect. 2.2]{Sta96}}]
	Let $\Delta$ be defined as above with
	$K[\Delta]/\Theta = (K[\Delta]/\Theta)_0 \oplus (K[\Delta]/\Theta)_1 \oplus \ldots \oplus (K[\Delta]/\Theta)_d.$
	Let $h(\Delta) = (h_0, \ldots, h_d)$ be the $h$-vector of $\Delta$.
	Then 
\[ 
\dim_K (K[\Delta]/\Theta)_i = h_i
\]
for $i= 0, \ldots, d$.
\end{proposition}

Let $\Delta$ be a $(d-1)$-dimensional Cohen--Macaulay complex with face ring 
$K[X_1 , \ldots , X_n ]/ I_{\Delta}= K[\Delta]$ and with l.s.o.p. $\Theta$.
An element $\omega \in K[X_1 , \ldots , X_n ]$ of degree $1$ is called a \emph{strong Lefschetz element} for $K[\Delta]/\Theta$ if the multiplication by $\omega^{d-2i}$,
\begin{align*}
\omega^{d-2i} : (K[\Delta]/\Theta)_i &\longrightarrow (K[\Delta]/\Theta)_{d-i}\\
m&\longmapsto \omega^{d-2i}m,
\end{align*}
is a bijection for $0 \le i \le \floor{\frac{d}{2}}$.

Following the notation from \cite{KN09}, we call
$\omega$ an \emph{almost strong Lefschetz element} for $K[\Delta]/\Theta$
if the multiplication 
by $\omega^{d-1-2i}$,
\begin{align*}
\omega^{d-1-2i} : (K[\Delta]/\Theta)_i &\longrightarrow (K[\Delta]/\Theta)_{d-1-i}\\
m&\longmapsto \omega^{d-1-2i}m,
\end{align*}
is an injection for $0 \le i \le \floor{\frac{d-1}{2}}$. A strong Lefschetz element is also an almost strong Lefschetz element because the multiplication by $\omega^{d-2i}$ is the composition of the multiplication with $\omega$, $d-2i$ times: if the resulting map is bijective, then the first map in the sequence has to be injective (and the last surjective).

A Cohen--Macaulay complex $\Delta$ is said to possess the \emph{strong Lefschetz property} if there exists a strong Lefschetz element for $K[\Delta]/\Theta$.
It follows from basic linear algebra that the existence of strong Lefschetz elements for Cohen--Macaulay complexes implies symmetry and unimodality of their $h$-vectors.
In particular the $h$-vectors of simplicial polytopes and simplicial spheres are unimodal by the following results. Stanley's result proves one direction in the $g$-theorem for simplicial polytopes.

\begin{theorem}[{Stanley \cite{Sta80}}]
	Boundary complexes of simplicial polytopes possess the strong Lefschetz property.
\end{theorem}

\section{Alcoved polytopes with interior lattice points} 
\subsection{The main theorem}
In this section, we give a proof of our main theorem.
\begin{theorem} \label{thm:interiorlatticepoints}
Let $P$ be a $d$-dimensional alcoved polytope with interior lattice points
such that every facet of $P$ has lattice distance $1$ to the set of interior lattice points.
Then its $h^*$-vector is unimodal. 
\end{theorem}

Before we move on to the proof, we discuss some consequences and computational experiments.
First, for even dimensions $d$, \Cref{thm:interiorlatticepoints} combined with the result by Hibi and Stanley discussed in the introduction (\cite[Theorem 1.3]{Ath04}) tells us that the peak always occurs at the middle entry $h^*_{\frac d 2 }(P)$.
For odd dimensions $d$, the two theorems only tell us that the peak occurs either at $h^*_{\frac{d-1}{2}}(P)$ or at $h^*_{\frac{d+1}{2}}(P)$.
In \Cref{sec:computations} we describe the algorithms that we used to generate random alcoved polytopes and calculate their $h^*$-vectors.
We tested around 20.000 alcoved polytopes of dimension up to 16. All $h^*$-vectors were unimodal. We tested around 1.000 alcoved polytopes with interior lattice points. In dimension greater than $5$, the $h^*$-vectors of all of these polytopes had their peak at $h^*_{\ceil{\frac{d-1}{2}}}(P)$. Interestingly, in dimension $5$, the peak also occured at $h^*_{\frac{d+1}{2}}(P)$.
See \Cref{alg:examples} for some examples of $h^*$-vectors of some randomly generated alcoved polytopes.

The first step towards the proof of our main theorem is the following statement.
\begin{theorem}[Adiprasito \& Steinmeyer \cite{Johanna}]\label{thm:UTunimodal}
	Let $P$ be a $d$-dimensional lattice polytope in $\R^d$ that admits a unimodular triangulation $\Delta$ such that the induced simplicial complex on $(\partial P)\cap \Z^d$ is a triangulation of $\partial P$. Then the $h^*$-vector of $P$ is unimodal.
\end{theorem}

\begin{proof}
	With Proposition \ref{BetkeMcMullen}, it is enough to show that the $h$-vector of the triangulation is unimodal.
	
	As a triangulated disc, $\Delta$ is a Cohen-Macaulay complex which can be extended to a triangulated sphere $\Sigma$, for instance by attaching a cone over the boundary of $\Delta$. So for a generic l.s.o.p.\ $\Theta$, any generic degree one element $\omega \in (K[\Sigma]/\Theta)_1$ is a strong Lefschetz element for $K[\Sigma]/\Theta$ by \cite[Theorem I{(1)}]{A18}. 
	It follows that $h_{\ceil{\frac{d+1}{2}}} \geq \ldots \geq h_d$ for the $h$-vector of~$\Delta$, compare \cite[Lemma 2.2]{STANLEY1993251}. To see this directly, consider for $j\leq (d+1)/2$ the commutative diagram
	\[
\begin{tikzcd}		
	(K[\Sigma]/\Theta)_j \arrow{r}{\ \cdot \omega^{d+1-2j}\ } \arrow[two heads]{d}{} & (K[\Sigma]/\Theta)_{d+1-j} \arrow[two heads]{d}{} \\
	(K[\Delta]/\Theta)_j\arrow{r}{\ \cdot \omega^{d+1-2j}\ } &  (K[\Delta]/\Theta)_{d+1-j},
\end{tikzcd}
\]
	where the vertical maps are surjective restriction maps and the top horizontal map is the Lefschetz isomorphism, implying surjectivity for the bottom horizontal map. In particular, the multiplication map $(K[\Delta]/\Theta)_{j} \xrightarrow{\cdot\omega} (K[\Delta]/\Theta)_{j+1}$ is surjective for all $j\geq (d+1)/2$.
	
	For the desired monoticity of the first half, we use \cite[Theorem 50]{adiprasito2021partition}, which gives us an injection
		\[
\begin{tikzcd}		
K[\Delta]/\Theta \arrow[hook]{r} &\bigoplus\limits_{v\in \mathrm{int}(\Delta)} K[\mathrm{star}(\Delta,v)]/\Theta .
\end{tikzcd}
\]
	Observe now that $\Theta$ is also an l.s.o.p.\ in $K[\mathrm{star}(\Delta,v)]$. By the Cone Lemma, see \cite[Theorem 7]{Lee} and \cite[Lemma 4.1]{A18}, we have an isomorphism of the rings
	$K[\mathrm{star}(\Delta,v)]/\Theta$  and $K[\mathrm{link}(\Delta,v)])/\tilde{\Theta},$ 
	where $\tilde{\Theta}$ is the projection of $\Theta$ to the orthogonal complement of $v$. 
	The link of any vertex $v\in\mathrm{int}(\Delta)$ is a sphere, so  we again get the generic Lefschetz property. 
	Combining the above into a commutative diagram, we obtain for any $j\leq d/2$:
	\[
	\begin{tikzcd}		
		(K[\Delta]/\Theta)_j \arrow{r}{\ \cdot \omega^{d-2j}\ } \arrow[hook]{d}{} & (K[\Delta]/\Theta)_{d-j} \arrow[hook]{d}{} \\
		\bigoplus\limits_{v\in \mathrm{int}(\Delta)}( K[\mathrm{star}(\Delta,v)]/\Theta)_j\arrow{r}{\ \cdot \omega^{d-2j}\ } & \bigoplus\limits_{v\in \mathrm{int}(\Delta)} (K[\mathrm{star}(\Delta,v)]/\Theta)_{d-j}
	\end{tikzcd}
	\]
	Here $\omega\in(K[\Delta]/\Theta)_1$ is chosen generically, and hence is generic in all restrictions  $ (K[\mathrm{star}(\Delta,v)]/\Theta)_1$. Thus, the bottom horizontal map is a Lefschetz isomorphism on each summand, and the top horizontal map is again injective. This gives us the injectivity of the multiplication $(K[\Delta]/\Theta)_{j-1} \xrightarrow{\cdot\omega} (K[\Delta]/\Theta)_{j}$ for all $j\leq d/2$, which implies the desired monotonicity $h_0\leq \ldots \leq h_{\floor{\frac{d+1}{2}}}$.
	\end{proof}

\begin{exm}
	Consider the alcoved polygon $P = [-1,1]^2$ and its alcoved triangulation $\Delta$. Then the assumptions of \Cref{thm:UTunimodal} are not satisfied. For instance, the edge between $(-1,0)$ and $(0,1)$ is present in the triangulation but supported on the boundary (see \Cref{fig:squares} on the left). So the induced simplicial complex on $(\partial P)\cap \Z^2$ is not a triangulation of $\partial P$.
\end{exm}

We can remedy the issue from the previous example by more carefully triangulating alcoved lattice polytopes under the assumption that the lattice distance of the facets to the interior lattice points is $1$.
\begin{lem}\label{lem:triangulation}
	Let $P$ be a $d$-dimensional alcoved polytope in $\R^d$ with interior lattice points such that every facet of $P$ has lattice distance $1$ to the set of interior lattice points. Then $P$ has a regular unimodular triangulation $\Delta$ such that the induced simplicial complex on $(\partial P)\cap \Z^d$ is a triangulation of $\partial P$. Moreover, the induced subdivision of this triangulation on any facet $F$ of $P$ is the alcoved triangulation of $F$.
\end{lem}

\begin{proof}
	We give a height function for the desired triangulation. Let $A$ be the set of lattice points in $P$. First, define the function $b\colon A \to \R$ sending every interior lattice point of $P$ to $0$ and every boundary lattice point of $P$ to $1$. This height function induces the subdivision of $A$ whose facets are the convex hull of the interior lattice points of $P$ and Cayley polytopes over the facets of $P$. For every facet of $F$ of $P$, the corresponding Cayley polytope is combinatorially equivalent to 
	\[ \conv\left( (F\times \{0\}) \cup (F'\times \{1\}) \right) \]
	where $F'$ is a set of interior lattice points at lattice distance $1$ from $F$, namely the face of $\conv(\mathrm{int}(P)\cap \Z^d)$ parallel to $F$.

	Second, define the height function 
	\[ a\colon  \left \{ 
		\begin{array}{l}
			A \to \R, \\
			(x_1,\ldots,x_d) \mapsto \sum_{i=1}^d x_i^2 +\sum_{\{i,j\}\in \{1, \ldots, d\}} (x_i-x_j)^2 
		\end{array}
	\right. \]
	that induces the alcoved triangulation of $P$. The height function of our desired triangulation is $h = b + \varepsilon a$ for sufficiently small $\varepsilon$ such that this height function induces a triangulation refining the subdivision induced by $b$. This regular triangulation is then unimodular by \cite[Lemma 4.15(2)]{HPPS17}. The induced simplicial complex on the boundary of $P$ is the boundary complex because the triangulation refines the subdivision induced by $b$. Since $b$ is constant on the boundary, the height function $h$ induces on every facet of $P$ the same subdivision as the height function $a$, which induces the alcoved triangulation.
\end{proof}

\begin{exm}
	The triangulation of the polygon $P = [-1,1]^2\subset \R^2$ given by the lemma is simply the cone over the boundary complex whose apex is the unique interior lattice polytope. The Cayley polytopes are the pyramids over the edges with respect to the interior lattice point and they are triangulated unimodularly in a different way compared to the alcoved triangulation. Both are shown in \Cref{fig:squares}.
\end{exm}

\begin{figure}
	\begin{tikzpicture}
		\node (p1) at (-6,-2) {};
		\node (p2) at (-4,-2) {};
		\node (p3) at (-2,-2) {};
		\node (p4) at (-6,0) {};
		\node (p5) at (-4,0) {};
		\node (p6) at (-2,0) {};
		\node (p7) at (-6,2) {};
		\node (p8) at (-4,2) {};
		\node (p9) at (-2,2) {};

		\filldraw[black] (p1) circle (4pt);
		\filldraw[black] (p2) circle (4pt);
		\filldraw[black] (p3) circle (4pt);
		\filldraw[black] (p4) circle (4pt);
		\filldraw[black] (p5) circle (4pt);
		\filldraw[black] (p6) circle (4pt);
		\filldraw[black] (p7) circle (4pt);
		\filldraw[black] (p8) circle (4pt);
		\filldraw[black] (p9) circle (4pt);
		\draw[very thick,black] (p1) -- (p2) -- (p3) -- (p6) -- (p9) -- (p8) -- (p7) -- (p4) -- (p1);
		\draw[thick,black!50!white] (p2) -- (p5) -- (p8);
		\draw[thick,black!50!white] (p4) -- (p5) -- (p6);
		\draw[thick,black!50!white] (p1) -- (p5) -- (p9);
		\draw[thick,black!50!white] (p4) -- (p8);
		\draw[thick,black!50!white] (p2) -- (p6);

		\node (q1) at (6,-2) {};
		\node (q2) at (4,-2) {};
		\node (q3) at (2,-2) {};
		\node (q4) at (6,0) {};
		\node (q5) at (4,0) {};
		\node (q6) at (2,0) {};
		\node (q7) at (6,2) {};
		\node (q8) at (4,2) {};
		\node (q9) at (2,2) {};

		\filldraw[black] (q1) circle (4pt);
		\filldraw[black] (q2) circle (4pt);
		\filldraw[black] (q3) circle (4pt);
		\filldraw[black] (q4) circle (4pt);
		\filldraw[black] (q5) circle (4pt);
		\filldraw[black] (q6) circle (4pt);
		\filldraw[black] (q7) circle (4pt);
		\filldraw[black] (q8) circle (4pt);
		\filldraw[black] (q9) circle (4pt);
		\draw[very thick,black] (q1) -- (q2) -- (q3) -- (q6) -- (q9) -- (q8) -- (q7) -- (q4) -- (q1);

		\draw[thick,black!50!white] (q1) -- (q5);
		\draw[thick,black!50!white] (q2) -- (q5);
		\draw[thick,black!50!white] (q3) -- (q5);
		\draw[thick,black!50!white] (q4) -- (q5);
		\draw[thick,black!50!white] (q6) -- (q5);
		\draw[thick,black!50!white] (q7) -- (q5);
		\draw[thick,black!50!white] (q8) -- (q5);
		\draw[thick,black!50!white] (q9) -- (q5);
	\end{tikzpicture}
	\caption{Two triangulations of the relexive square: the alcoved triangulation on the left and the triangulation from \Cref{lem:triangulation} on the right.}
	\label{fig:squares}
\end{figure}
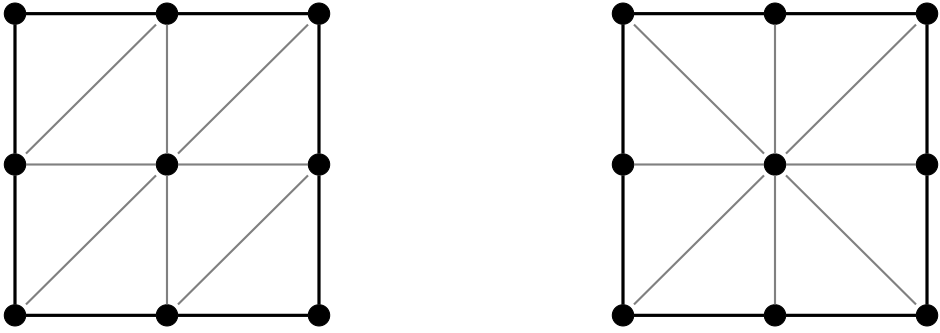

\begin{rem}\label{rem:proofstrategy}
	The assumption that the facets have lattice distance $1$ to the interior lattice points in \Cref{lem:triangulation} is necessary. If an alcoved polytope $P$ has a facet $F$ that has lattice distance at least $2$ to the interior lattice points, then $P$ cannot have any unimodular triangulation $\Delta$ such that the induced simplicial complex on the boundary points is the boundary complex. Indeed, any such triangulation that is compatible with the boundary has a simplex over a simplex in the facet $F$ (full-dimensional relative to $F$) that has height at least $2$ and is therefore not unimodular.
\end{rem}

\begin{proof}[{Proof of \Cref{thm:interiorlatticepoints}}]
	Combine \Cref{thm:UTunimodal} and \Cref{lem:triangulation}.
\end{proof}

\subsection{The lattice distance from the interior lattice points}
Since the $h^*$-vector of the alcoved polytopes are the $h$-vectors of the alcoved triangulations, it follows that $h^*_i(P)\ge h^*_i(P')$ for all $i \in \{0, \ldots, d\}$. A good approximation of $P$ by $P'$ gives a good approximation of $h^*(P)$ by $h^*(P')$. It is therefore interesting to know how well $P'$ approximates $P$. The next theorem tells us how far a facet can be from the interior lattice points.
\begin{theorem} \label{thm:distance}
Let $P$ be a $d$-dimensional alcoved polytope with interior lattice points. Then the maximal lattice distance of a facet of $P$ to the interior lattice points is $d-1$.
\end{theorem}
\begin{proof}
Let $F$ be a facet of $P$. 
Let $P'$ be the $d$-dimensional alcoved polyhedron obtained by removing the
facet-defining hyperplane of $F$
from the hyperplane description of $P$.
We distinguish between two different cases:
Either $P'$ is an unbounded polyhedron or a polytope.

Case 1.
If $P'$ is an unbounded polyhedron, then $F$ has lattice distance $1$ to the interior lattice points.
To see this, observe that the
recession cone $\mathcal{C}$ of $P'$ is an alcoved cone, i.e. an affine cone that is an alcoved polyhedron. 
The intersection of an alcoved polyhedron with alcove hyperplanes (hyperplanes parallel to facets of $Q_d$) 
is again an alcoved polyhedron, any possible vertices have to be lattice points by Definition~\ref{def:alcoved}.
Let $x$ be an interior lattice point of $P$ (and hence of $P'$). Let $\mathcal{C}'$ denote the translate of $\mathcal{C}$
with apex $x$. $\mathcal{C}'$ is contained in the interior of $P'$. Let $H$ be the hyperplane parallel to $F$ that has distance 1 from $x$ and separates $x$ and $F$. The intersection of $\mathcal{C}'$ with $H$ is a lattice polytope
contained in the interior of  $P'$. If $F$ has distance larger than $1$ from $x$, then  $\mathcal{C}' \cap H$ is contained in the interior of $P$ and its vertices are interior lattice points of $P$ with smaller lattice distance to $F$ than the distance between $x$ and $F$. This shows that $F$ has lattice distance $1$ to the interior lattice points of $P$.

 \begin{figure}[htbp]
\centering
 \includegraphics[width=.7\linewidth]{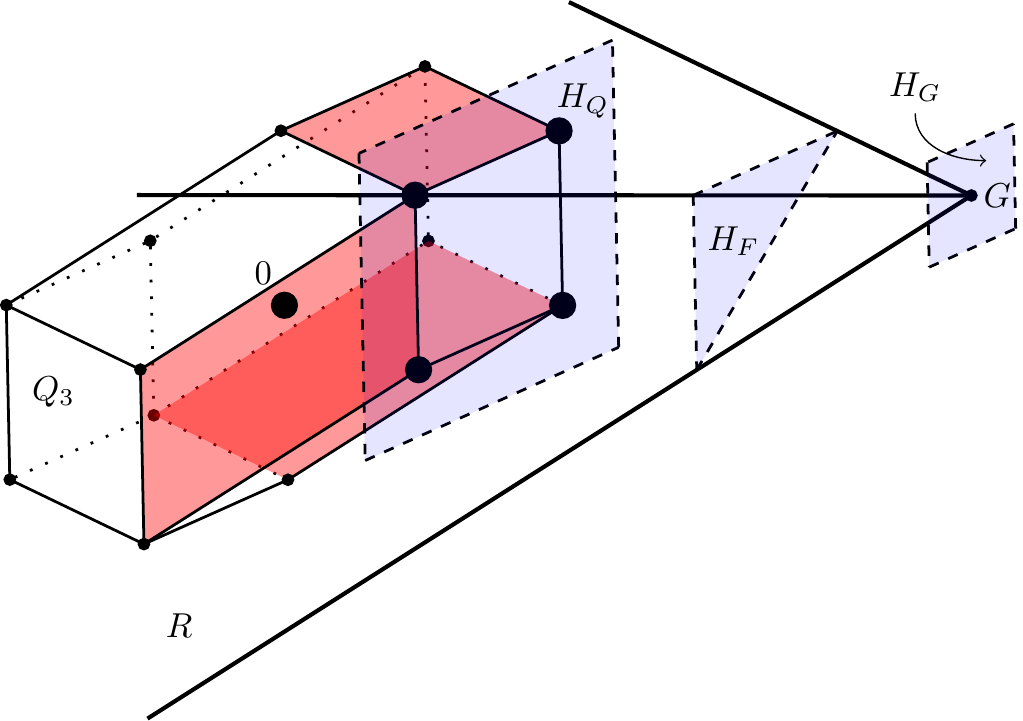}
\caption[Facet cannot be too far from the interior lattice points]{The hyperplanes containing the red facets of $Q_3$ intersect in the affine hull of face $G$ of $R$.}
\label{fig:propdist}
\end{figure}

Case 2. 
See Figure~\ref{fig:propdist} for an example in dimension 3.
Let $x$ be an interior lattice point of $P$ closest to $F$. We may assume that $x=0$.
 Let $H_F$ be the hyperplane containing $F$.
The polytope $Q_d$ from Definition~\ref{def:Qd} is contained in all alcoved $d$-polytopes which contain the origin in the interior. 
Any facet of an alcoved $d$-polytope is parallel to two facets of $Q_d$.
Among the two facet-defining hyperplanes of $Q_d$ which are parallel to $F$, let $H_Q$ denote the one separating $0$ and $F$.
 The vertices of $Q_d$ on $H_Q$ are lattice points in $P$ which are closer to $F$ than $0$. Since $0$ is closest to $F$ among all interior lattice points of $P$, these lattice points have to be in the boundary of $P$, each of the points has to be contained in at least one facet of $P$.
 Consider only the facets of $P$ containing the vertices of $Q_d$ in $H_Q$ and (additionally) facet $F$.
 The facet-defining half-spaces of these facets define an alcoved polyhedron with $0$ in the interior.
 If the polyhedron $R$ obtained by removing the facet-defining half-space of $F$ 
 from the hyperplane description of the polyhedron
 is unbounded in the direction of the facet normal of $F$, 
 then by case 1 $H_F$ (and hence $F$) has lattice distance 1 from $0$.
 Assume the polyhedron $R$ is bounded in direction of the facet normal of $F$.
Then there is a hyperplane $H_G$ parallel to $H_F$ which intersects $R$ in a $k$-face $G$ of $R$ and such that $H_F$ separates $x$ and $H_G$.
If $H_Q= H_G$, then $H_F=H_Q$, and $H_F$ has distance 1 from $0$.
Assume $H_G$ is not equal to $H_Q$. Then $H_G$ is not equal to $H_F$ either.
We will show that $H_G$ has lattice distance at most $d$ from $0$,
and therefore $H_F$ has at most lattice distance $d-1$ from $0$.
$G$ is given as an intersection of $d-k$ facets of $R$.
The hyperplane $H_G$ containing $G$ and parallel to $F$ is of the form $\{x\in \R^d\mid x_i-x_j=l\}$ for some $i, j \in\{0, \ldots, d\}$
with $i\ne j$ and $x_0\coloneqq 0$ and for some positive integer $l$.

We know that the difference $x_i-x_j=l$ is defined from some equations of the form $x_s-x_t=1$,
 for $s,t\in \{0, \ldots, d\}$ and $s\ne t$.
So the difference $l$ between the two variables $x_i$ and $x_j$ can be obtained from setting the difference between some pairs of $d-1$ variables to 1.
This shows that $l\in \{0,\ldots, d\}$.

We can also state this as a graph theoretical problem:
Let $G$ be a simple graph on $d+1$ vertices $v_0, \ldots, v_{d+1}$.
There is an edge between vertex $v_s$ and vertex $v_t$ if and only if 
either $\{x\in \R^d\mid x_s-x_t=1\}$ or $\{x\in \R^d\mid x_t-x_s=1\}$ is a hyperplane of $R$ intersecting in face $G$.
If $G$ would contain a cycle $(v_{s_1},v_{s_2}),(v_{s_2},v_{s_3}),\ldots, (v_{s_{r-1}},v_{s_r}), (v_{s_r},v_{s_1})$, then 
$x_{s_1} > x_{s_2} >\ldots > x_{s_r} >x_{s_1}$, a contradiction.
So $G$ does not contain any cycles, 
it is a forest.
The condition that $R$ is bounded in the direction of the facet-normal of $F$
translates to the condition that vertex $x_i$ and vertex $x_j$ are path-connected.
The longest possible path-length in a forest on $d+1$ vertices is $d$.
The difference $l$ is therefore at most $d$ and facet $F$ has lattice distance at most $d-1$ to $0$.
\end{proof}

We end this section with a proposition that shows that
the bound from Theorem~\ref{thm:distance} is sharp for all dimensions $d$. 

\begin{proposition} \label{prop:example_sharpbound} 
There is a $d$-dimensional  alcoved polytope for any $d\in \N$ which has a facet with lattice distance $d-1$ to the interior 
lattice points.
\end{proposition}
\begin{proof}
We construct an example from the following polytope:
Let $P$ be the  scaling of the order polytope of 
 the chain of length $d$ by $(d+1)$:
\begin{align*}
P= \{x\in \R^d: x_1&\ge 0,\\
x_1 &\leq x_2, \\
x_2 &\leq x_3, \\
&\vdots \\
x_{d-1} &\leq x_d,\\
x_d &\le d+1\}. 
\end{align*}
$P$ is  a reflexive alcoved simplex. The vertex description of $P$ is given by
\begin{align*}
P=\conv\{&(0,0,\ldots, 0,0),(0,0, \ldots, 0, d+1),\ldots, \\
&(0,d+1, \ldots, d+1, d+1),(d+1,d+1, \ldots, d+1, d+1)\}.
\end{align*}
Its unique interior lattice point is $p=(1,2,3, \ldots, d)$.
The polytope 
$P \cap \{x\in \R^d\mid x_1 \le d\}$ is still an alcoved polytope with unique interior lattice point $p$.
The  hyperplane $\{x\in \R^d\mid x_1 =d\}$ defining the new facet has distance $d-1$ from point $p$.
\end{proof}

For a visualization in dimension $3$, see Figure~\ref{fig:exampledist}.
\begin{figure}[htbp]
\centering
 \includegraphics[width=.9\linewidth]{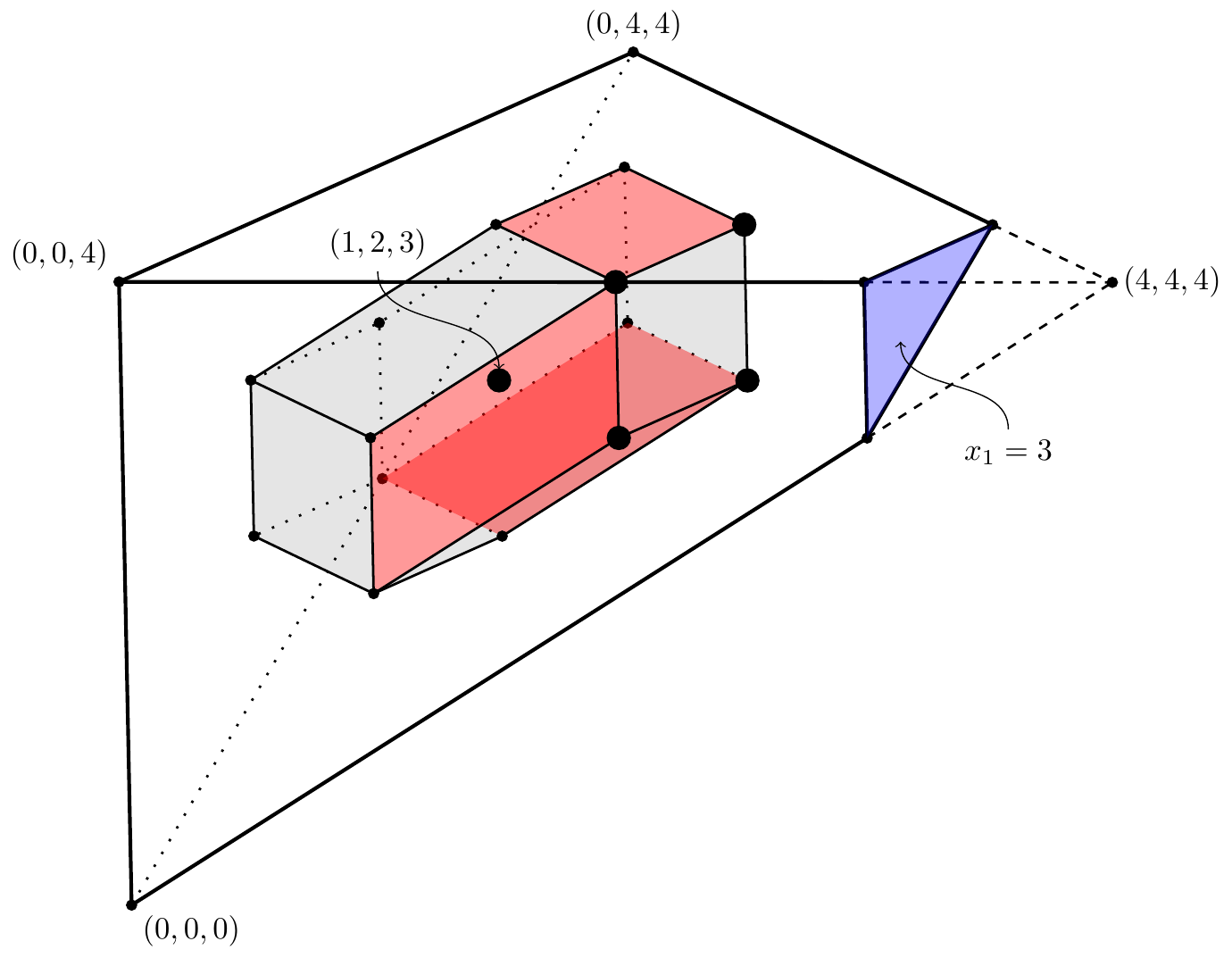}
\caption[Facet with distance $d-1$ to interior lattice points]{The blue facet has lattice distance $2$ from the interior lattice point. The red facets of $Q_3$ 
are contained in facets of the bigger polytope.}
\label{fig:exampledist}
\end{figure}

\section{Computational Experiments}\label{sec:computations}
\lstset{language=Python,showstringspaces=false,breaklines=true} 
Here we completely list the algorithms used to calculated $h^*$-vectors of polytopes and test for unimodality that give the results shown in the last subsection. The code is complete so that readers can run experiments by copy-paste. All algorithms are written as SageMath code \cite{SageMath}.
We also give some examples of alcoved polytopes and their $h^*$-vectors.

%%%%%%%%%%%%%%%%%%%%%%%%%%%%%%%%%%%%%%%%%
\subsection{\texorpdfstring{Convert Ehrhart polynomial to $h^*$-polynomial}{Convert Ehrhart polynomial to h{*}-polynomial}}\label{alg:ehrconv}
%%%%%%%%%%%%%%%%%%%%%%%%%%%%%%%%%%%%%%%%%%
The following program was written in collaboration with Sophia Elia. 
This program converts the Ehrhart polynomial of a lattice polytope to the $h^*$-polynomial.
The Ehrhart  polynomials of lattice polytopes can be calculated  using LattE Integrale \cite{Latte}.
There is also a built-in normaliz function \cite{Normaliz} to calculate the Ehrhart series of a polytope, but its run-time was too long for our examples.

The approach for this function is based on Eulerian polynomials (see~\cite[Section~2.2]{BeRo07} The Ehrhart series can be rewritten as follows:
\[ \Ehr_P(t) = \sum_{m\geq 0} L_P(m)t^m = \sum_{m\geq 0}\sum_{j \geq 0}^{d}c_j m^j t^m,\]
where we write the Ehrhart polynomial as $L_P(m) = \sum_{j \geq 0} c_j m^j$. The $h^*$-polynomial is defined as $(1-t)^{d+1} \Ehr_P(t)$. 
We now exploit that the Eulerian numbers $A(d,k)$ can be defined by the generating function expression
\[
	\sum_{j\geq 0} j^d z^j = \frac{\sum_{k=0}^d A(d,k) z^k}{(1-z)^{d+1}}.
\]
So after some calculation, we can write the $h^*$-polynomial as
\[
	h^*_P(z) = \sum_{i=0}^d c_i \left( \sum_{k=0}^d A(i,k) z^k\right) (1-z)^{d-i}.
\]

We begin with a function that outputs the Eulerian numbers up to $A(n,n)$ arranged in an $(n+1)\times (n+1)$ matrix.
\begin{lstlisting}
def eulerian_numbers(n):
    A = zero_matrix(n+1,n+1)
    A[0,0] = 1
    for i in range(1,n+1):
        A[i,0] = 0
        A[i,1] = 1
    for j in range(2,n+1):
        for k in range(2,j+1):
            if j == k:
                A[j,k] = 1
            else:
                A[j,k] = (j-k+1)*A[j-1,k-1] +k*A[j-1,k]
	return(A)
\end{lstlisting}

The following function makes the Eulerian polynomials based on the previous function computing the Eulerian numbers.
\begin{lstlisting}
def eulerian_polynomial(n):
    R = PolynomialRing(ZZ, 't')
    t = R.gen()
    A = eulerian_numbers(n)
	return(R.sum( A[n,i]*t**i for i in range(n+1)))
\end{lstlisting}

Finally, here is the function that converts the Ehrhart polynomial to the $h^*$-polynomial. As input, it takes \texttt{ehr\_poly}, a polynomial in \texttt{t} with rational coefficients. The output is the corresponding $h^*$-polynomial.
\begin{lstlisting}
def ehr_to_hstar(ehr_poly):
    # change the polynomial into a vector
    Ring = PolynomialRing(QQbar, 't')
    t = Ring.gen()
    ehr_poly = ehr_poly.coefficients()
    # get the dimension of the polytope
    d = len(ehr_poly)-1
    # compute the h* polynomial
    factors = zero_vector(d+1)
    factors = factors.change_ring(Ring)
    for j in range(d+1):
        factors[j] = ehr_poly[j]*(1-t)**(d-j)*eulerian_polynomial(j)
    return sum(factors)
\end{lstlisting}

As an example, the following code verifies that the $h^*$-polynomial of a unimodular simplex is indeed $1$.
\begin{lstlisting}
    sage: p = polytopes.simplex(4)
    sage: e = p.ehrhart_polynomial()
    sage: ehr_to_hstar(e)
\end{lstlisting}

%%%%%%%%%%%%%%%%%%%%%%%%%%%%%%%%%%%%%%%%%
\subsection{Random alcoved polytopes}\label{alg:ehrrand}
%%%%%%%%%%%%%%%%%%%%%%%%%%%%%%%%%%%%%%%%%
We give an example of a function which allows us to create random alcoved polytopes. The polytopes are full-dimensional and inscribed in a cube of dimension $\dim(P)$ and with edge-length $5$
(or smaller). This is our way of making sure that the volume does not get too large for computation. 

The first function organizes a list of numbers into an $\mathcal{H}$-description of an alcoved polytope. As input, it takes the dimension \texttt{dim} of the polytope and a vector \texttt{vec} of length $2 \cdot \binom{\texttt{dim}}{2} + 2 \cdot \texttt{dim}$. The output is a matrix of size $\left(2 \cdot \binom{\texttt{dim}}{2} + 2 \cdot \texttt{dim}\right)\times (\texttt{dim}+1)$ for the $\mathcal{H}$-description.
\begin{lstlisting}
def alcoved_matrix(dim, vec):
    c = binomial(dim,2)
    M_help = Matrix (2*c,dim)
    M = Matrix (2*c+2*dim,dim+1)
    # 
    # all hyperplanes of type x_i-x_j = constant:
    for i in range(c):
    # all sets of 2 indices out of all indices for each choice of (x_i,x_j)
        pairij = Combinations(range(dim),2).list()[i]
        M_help[2*i, pairij[0] ] = 1
        M_help[2*i,pairij[1] ] = -1
        M_help[2*i+1, pairij[0] ] = -1
        M_help[2*i+1, pairij[1] ] = 1
    for i in range(2*c):
        M[i, 0 ] = vec[i]
        for j in range(dim): M[i,j+1] = M_help[i,j]
    # all hyperplanes of type +x_i = constant
    # and -x_i = constant
    for i in range(2*c, 2*c+dim):
        M[i, 0] = vec[i]
        M[i+dim, 0] = vec[i+dim]
        M[i, i+1-2*c]=-1
        M[i+dim, i+1-2*c]=1
    return M
\end{lstlisting}

The following function produces a random input vector for the previous function (satisfying the constraint that the produced alcoved polytope is contained in the appropriate cube). Its input is just the desired dimension \texttt{dim} of the alcoved polytope. The output is a vector of length $2\cdot \binom{\texttt{dim}}{2} + 2\cdot \texttt{dim}$ that is suitable as the input \texttt{vec} for the previous function. The alcoved polytope generated by these two functions always contains the cube $[0,1]^\texttt{dim}$ and is contained in the cube $[-2,3]^\texttt{dim}$. 
\begin{lstlisting}
def random_vector(dim):
    rand=random_matrix(ZZ, 1, 2*binomial(dim,2),  
    x=1,y=6).augment(random_matrix(ZZ, 1, dim, x=1,
    y=4).augment(random_matrix(ZZ, 1, dim, x=0, y=3)))[0]
	return rand
\end{lstlisting}

Alternatively, for large dimension, we might want an alcoved polytope of smaller volume. In this case, we produce a random alcoved polytope in the following way. Again, this function takes the desired dimension \texttt{dim} as input and outputs a vector of length $2\cdot \binom{\texttt{dim}}{2} + 2\cdot \texttt{dim}$ that is suitable as the input \texttt{vec} for the previous function. The alcoved polytope generated by these two functions always contains the cube $[0,1]^\texttt{dim}$ but it is now contained in the cube $[0,3]^\texttt{dim}$. 
\begin{lstlisting}        
 def small_random_vector(dim):
    rand=random_matrix(ZZ, 1, 2*binomial(dim,2),   x=1,y=4).augment(random_matrix(ZZ, 1, dim, x=1,
    y=3).augment(random_matrix(ZZ, 1, dim, x=0, y=2)))[0]
    return rand
\end{lstlisting}

Finally, the following function combines the previous functions to make the polytope. As input, it takes the desired dimension \texttt{dim} of the alcoved polytope and it takes a vector of length $2\cdot \binom{\texttt{dim}}{2} + 2\cdot \texttt{dim}$. It then makes the alcoved polytope by building its $\mathcal{H}$-description using the function shown above.
\begin{lstlisting}
def alcoved_polytope(dim,vec):
    P = Polyhedron(ieqs = alcoved_matrix(dim, vec), backend='normaliz')
    return P
\end{lstlisting}

%%%%%%%%%%%%%%%%%%%%%%%%%%%%%%%%%%%%%%%%%
\subsection{Unimodality}\label{alg:unimodality}
%%%%%%%%%%%%%%%%%%%%%%%%%%%%%%%%%%%%%%%%%
The programs in this section are 
a function that determines whether a list is unimodal and 
a function that tests the $h^*$-vectors of a given number of randomly generated alcoved polytopes
of a given dimension for unimodality.\\
First we convert the  $h^*$-polynomial to the  $h^*$-vector:
\begin{lstlisting}
# h^*-vector from h^*-polynomial 
def Hvec(x): return (ehr_to_hstar(x.ehrhart_polynomial())).coefficients() 
\end{lstlisting}

Then we determine if a list or tuple is unimodal using the following function. The input is a list or tuple \texttt{ls} of numbers. The output is \texttt{True} if the list is unimodal and \texttt{False} otherwise.
\begin{lstlisting}
def unimodal(ls):
    # determines index of the first decrease in the list
    decr= next((i for i in range(1,len(ls)) if ls[i]-ls[i-1]<0),False)
    if (decr == False) or (decr == len(ls)):
       return True
    else: 
    # determines if there is an increase after the decrease
         if any(ls[i]-ls[i-1]>0 for i in range(decr+1,len(ls))):
           return False
         else: return True
\end{lstlisting}

Next we test several $h^*$-vectors of alcoved polytopes for unimodality. The function randomly generates alcoved polytopes specified by the input \texttt{num\_of\_tests} and \texttt{dim}. 
The first input is the number of random instances that the function generates. The second input determines the dimension. The output is the first polytope with non-unimodal $h^*$-vector if such a polytope is found. The output is then a vector that can be turned into a polytope with the \texttt{alcoved\_polytope}-function described above. This function is shown here with the \texttt{random\_vector}-function. It can, of course, also be used with the \texttt{small\_random\_vector}-function instead.
\begin{lstlisting}
def test_for_unimodality(num_of_tests, dim):
    for i in range(num_of_tests):
        rv = random_vector(dim)
        P = alcoved_polytope(dim,rv)
        hvector = Hvec(P)
        if unimodal(hvector) == False:
           print('There is an alcoved polytope with non-unimodal h-star-vector!')
           return rv
    print('All {} tested polytopes have unimodal h-star-vector.'.format(num_of_tests))
\end{lstlisting}
%%%%%%%%%%%%%%%%%%%%%%%%%%%%%%%%%%%%%%%%%
\subsection{Examples}\label{alg:examples}
%%%%%%%%%%%%%%%%%%%%%%%%%%%%%%%%%%%%%%%%%
In Table~\ref{table:hvectors} there are some examples of (unimodal) $h^*$-vectors
calculated with the function \verb|Hvec| coming from alcoved polytopes randomly generated  with the functions described in \Cref{alg:ehrrand}.

There is one example for each dimension between 3 and 16.
For dimension 16 we used the \verb|small_random_vector|
to generate the polytopes. For smaller dimension, we used \verb|random_vector|.

\begin{table}
	\centering
	\begin{tabular}{l|l|l}
		\begin{tabular}[t]{c|l}
			Dim. & $h^*$-vector \\ \hline\hline
			$3$&$1,98,188,22$\\
			\hline
			$4$&$1,104,370,146,3$\\
			\hline
			$5$&$1,356,3216,3965,722,6$\\
			\hline
			$6$&$1,278,3442,7074,2977,189$\\
			\hline
			&$1,522,9978,33062,$\\ 
			$7$&$26037,4584,102$\\
			\hline
			&$1,1962,89574,650410,$\\
			$8$&$1210219,643554,86807,1762$\\ \hline
		    &$1,$\\
			&$5624,$\\ 
			&$409384,$\\ 
			&$4544444,$\\
			$9$&$13413250,$\\
			&$12480018,$\\
			&$3617286,$\\
			&$265589,$\\
			&$2485$\\
			\hline
			&$1,$\\
			&$11361,$\\
			&$1472932,$\\
			&$26848012,$\\
			$10$&$130799314,$\\
			&$213186186,$\\
			&$123024451,$\\
			&$23477417,$\\
			&$1154189,$\\
			&$6839$\\
			\hline
			&$1,$\\
			&$14350,$\\
			&$2256181,$\\
			&$51639116,$\\
			&$327390062,$\\
			$11$&$729687382,$\\
			&$621049446,$\\
			&$198471087,$\\
			&$20964122,$\\
			&$521849,$\\
			&$1070$\\ 
		\end{tabular}
		&
		\begin{tabular}[t]{c|l}
			&\\
			&$1,$\\
			&$19944,$\\
			&$5228717,$\\
			&$186169763,$\\
			&$1812852282,$\\
			$12$&$6346402583,$\\
			&$8965318338,$\\
			&$5257171069,$\\
			&$1223272752,$\\
			&$98031585,$\\
			&$1936661,$\\
			&$3716$\\
			\hline
			&$1,$\\
			&$91329,$\\
			&$44855680,$\\
			&$2470235309,$\\
			&$35303767765,$\\
			&$180399454984,$\\
			$13$&$380503494555,$\\
			&$350270955412,$\\
			&$139818187839,$\\
			&$22527990348,$\\
			&$1233354464,$\\
			&$15460541,$\\
			&$13816$\\
			\hline
			&$1,$\\
			&$92744,$\\
			&$61035357,$\\
			&$4461610156,$\\
			&$84498080329,$\\
			&$579384773749,$\\
			$14$&$1683186967226,$\\
			&$2226507192578,$\\
			&$1364315287619,$\\
			&$375282480730,$\\
			&$42267367262,$\\
			&$1603500992,$\\
			&$13275503,$\\
			&$6792$\\
		\end{tabular}
		&
		\begin{tabular}[t]{c|l}
		    &\\
			&$1,$\\
			&$240485,$\\
			&$278806016,$\\
			&$30713524660,$\\
			&$831462306799,$\\
			&$8037261688378,$\\
			&$33174929171015,$\\
			$15$&$63952385251690,$\\
			&$59735193988116,$\\
			&$26931894294332,$\\
			&$5599149137353,$\\
			&$484130852634,$\\
			&$14170655809,$\\
			&$90022039,$\\
			&$36218$\\
			\hline	
			&$1,$\\
			&$93896,$\\
			&$91469858,$\\
			&$10614760282,$\\
			&$323897816139,$\\
			&$3658410128555,$\\
			&$18169693219004,$\\
			$16$&$43468545132385,$\\
			&$52316695445712,$\\
			&$31944633402938,$\\
			&$9656730519861,$\\
			&$1355225501529,$\\
			&$77874204105,$\\
			&$1436104725,$\\
			&$4957986,$\\
			&$586$\\ 
		\end{tabular}
	\end{tabular}
	\caption{Examples of $h^*$-vectors}
	\label{table:hvectors}
\end{table}

\bigskip

\noindent \textbf{Acknowledgements.}
We would like to thank Christian Haase for asking the question as well as Martina Juhnke-Kubitzke and Johanna Steinmeyer for educational discussions and fruitful hints. We also profitted from the MATH+ Thematic Einstein Semester Algebraic Geometry: Varieties, Polyhedra, Computation.

\newcommand{\SortNoop}[1]{}

\end{document}